\documentclass[11pt]{amsart}
\usepackage[T1]{fontenc}
\usepackage{lmodern}
\usepackage{amssymb,mathtools}
\usepackage[expansion=false]{microtype}
\usepackage[hidelinks]{hyperref}

\newtheorem{theorem}{Theorem}[section]
\newtheorem{lemma}[theorem]{Lemma}

\numberwithin{equation}{section}
\newcommand{\R}{\mathbb R}
\newcommand{\A}{\mathcal A}
\newcommand{\1}{\mathbf 1}
\newcommand{\norm}[1]{\left\lVert#1\right\rVert}
\DeclareMathOperator{\dist}{dist}

\title[Endpoint estimates for bilinear circular averages]
{Endpoint estimates for the bilinear circular averaging operator}
\author{Xuyi Wei}
\address{School of Mathematical Sciences, Beihang University, Beijing, China}
\email{weixuyi@buaa.edu.cn}
\date{}
\hypersetup{
 pdftitle={Endpoint estimates for the bilinear circular averaging operator},
 pdfauthor={Xuyi Wei}
}

\begin{document}
\begin{abstract}
We prove that the bilinear circular averaging operator is bounded from
$L^1(\R)\times L^1(\R)$ to $L^{1/2,\infty}(\R)$.
The corresponding strong estimate fails, even when the inputs are
nonnegative, smooth, and compactly supported. We also construct
nonnegative $f,g\in L^1(\R)$ for which the average does not belong to
$L^{1/2}(\R)$. The weak estimate is proved by a dyadic decomposition
of the inputs and a summation over interval pairs near each quadratic fold.
\end{abstract}
\maketitle

\section{Introduction}

For measurable functions $f,g$ on $\R$, let
\begin{equation}\label{eq:operator}
 \A(f,g)(x)=\int_0^{2\pi}
 f(x-\cos\theta)g(x-\sin\theta)\,d\theta.
\end{equation}
Here $d\theta$ is unnormalized arclength measure. Replacing it by
$d\theta/(2\pi)$ affects only the constants in the estimates.
For $0<p<\infty$, write
\[
 \norm{H}_{L^p}=\left(\int_{\R}|H(x)|^p\,dx\right)^{1/p},
 \qquad
 \norm{H}_{L^{p,\infty}}
 =\sup_{\lambda>0}\lambda
   |\{x:|H(x)|>\lambda\}|^{1/p}.
\]

Multilinear convolutions against spherical measure were studied by
Oberlin~\cite{Oberlin} and, at the endpoints, by Bak and
Shim~\cite{BS}. Shrivastava and Shuin~\cite{SS} characterized the
boundedness region in the Banach range. For related bounds for lacunary
bilinear spherical maximal operators, see~\cite{BF}.
Christ and Zhou~\cite[Lemma~2.2]{CZ} proved
\[
 \A:L^p(\R)\times L^q(\R)\longrightarrow L^{pq/(p+q)}(\R),
 \qquad 1<p,q\leq\infty,
\]
with the usual interpretation when an exponent is infinite.
For bilinear spherical averages in dimensions $d\geq2$, the strong
$L^1\times L^1\to L^{1/2}$ estimate was proved in~\cite{IPS}.
Borges~\cite[Question~5.2]{Borges} asks whether this estimate also holds
in dimension one. The following theorem answers this question negatively
and establishes the corresponding weak bound.

\begin{theorem}\label{thm:main}
There is an absolute constant $C$ such that, for all $f,g\in L^1(\R)$
and $\lambda>0$,
\begin{equation}\label{eq:weak-distribution}
 |\{x:\A(|f|,|g|)(x)>\lambda\}|
 \leq C\left(\frac{\norm{f}_1\norm{g}_1}{\lambda}\right)^{1/2}.
\end{equation}
In particular, the integral defining $\A(f,g)$ converges absolutely
almost everywhere, and
\begin{equation}\label{eq:weak-norm}
 \norm{\A(f,g)}_{L^{1/2,\infty}}
 \leq C^2\norm{f}_1\norm{g}_1.
\end{equation}
There are nonnegative compactly supported $f,g\in L^1(\R)$, both of
norm one, for which
\begin{equation}\label{eq:strong-infinite}
 \int_{\R}\A(f,g)(x)^{1/2}\,dx=\infty.
\end{equation}
Moreover, there are nonnegative $f_n,g_n\in C_c^\infty(\R)$ with
$\norm{f_n}_1=\norm{g_n}_1=1$ and
$\norm{\A(f_n,g_n)}_{L^{1/2}}\to\infty$.
\end{theorem}

The Jacobian of the map $(x,\theta)\mapsto
(x-\cos\theta,x-\sin\theta)$ vanishes precisely at
$\theta=3\pi/4$ and $7\pi/4$. Away from these angles, a change of
variables gives the strong $L^{1/2}$ estimate. At either angle, the
difference between the two coordinates has a nondegenerate critical
point. The local weak estimate for this geometry is proved in
Section~\ref{sec:fold} and applied to the circle in
Section~\ref{sec:weak}. The counterexample is given in
Section~\ref{sec:strong}.

Nonnegative integrals are understood in the extended Lebesgue sense.
We choose Borel representatives of $L^1$ functions; the averages are
independent of this choice almost everywhere. Indeed, if $E$ is a null
set and $b$ is a measurable function on a finite interval $J$, then
Tonelli's theorem gives
\[
 \int_{\R}\int_J\1_E(x+b(t))\,dt\,dx=|J|\,|E|=0.
\]
The letter $C$ denotes a finite constant whose value may change.
For an interval $I$, $\ell(I)$ is its length and $3I$ is the concentric
interval of three times its length.

\section{A weak estimate near a quadratic fold}\label{sec:fold}

\begin{lemma}\label{lem:fold}
Let $\delta>0$ and $\phi,\chi\in C^1([-\delta,\delta])$.
Put $\psi=\chi-\phi$. Suppose that there are constants $c_0,C_0>0$
such that, for $0<|t|\leq\delta$,
\begin{align}
 |\phi(t)|+|\chi(t)|&\leq C_0|t|,\label{eq:fold-linear}\\
 c_0t^2\leq\psi(t)&\leq C_0t^2,\label{eq:fold-quadratic}\\
 c_0|t|\leq|\psi'(t)|&\leq C_0|t|.\label{eq:fold-derivative}
\end{align}
Define, for nonnegative $F,G\in L^1(\R)$,
\[
 T(F,G)(x)=\int_{-\delta}^{\delta}
 F(x+\phi(t))G(x+\chi(t))\,dt.
\]
Then
\begin{equation}\label{eq:fold-weak}
 |\{x:T(F,G)(x)>\lambda\}|
 \leq C\left(\frac{\norm{F}_1\norm{G}_1}{\lambda}\right)^{1/2},
 \qquad\lambda>0,
\end{equation}
where $C$ depends only on $c_0,C_0,\delta$.
\end{lemma}

\begin{proof}
If either input has zero norm, the assertion follows from Tonelli's
theorem. By dividing the inputs by their respective norms and
rescaling $\lambda$, it suffices to assume
\begin{equation}\label{eq:normalization}
 \norm{F}_1=\norm{G}_1=1.
\end{equation}
Fix $\lambda>0$ and set
\begin{equation}\label{eq:alpha}
 \alpha=(8\delta)^{-1/2}\lambda^{1/2}.
\end{equation}

\smallskip\noindent
\emph{Dyadic decomposition.}
We apply the dyadic Calder\'on--Zygmund
decomposition~\cite[Chapter~I]{Stein} while keeping all pieces nonnegative.
Let $\mathcal I$ be the maximal dyadic intervals on which the average
of $F$ exceeds $\alpha$. Such maximal intervals exist because the
average over an interval of length tending to infinity tends to zero.
They are pairwise disjoint. Set
\[
 a_I=\int_I F(u)\,du,\qquad I\in\mathcal I,
\]
and note that the average over the dyadic parent of $I$ is at most
$\alpha$. Hence
\begin{equation}\label{eq:CZ-masses}
 \alpha\ell(I)<a_I\leq2\alpha\ell(I),\qquad
 \sum_{I\in\mathcal I}a_I\leq1,\qquad
 \sum_{I\in\mathcal I}\ell(I)\leq\alpha^{-1}.
\end{equation}
Write
\[
 F=F_0+F_b,\qquad
 F_0=F\1_{\R\setminus\bigcup\mathcal I},\qquad
 F_b=\sum_{I\in\mathcal I}F_I,\qquad F_I=F\1_I.
\]
Dyadic differentiation gives $0\leq F_0\leq\alpha$ almost everywhere.
Applying the same construction to $G$ gives a disjoint family
$\mathcal J$, functions $G_0,G_b,G_J$, and masses $b_J=\int_JG$
satisfying the analogues of \eqref{eq:CZ-masses}. Set
\begin{equation}\label{eq:exceptional}
 \Omega=\bigcup_{I\in\mathcal I}3I
       \,\cup\!\bigcup_{J\in\mathcal J}3J.
 \qquad\text{Then}\qquad |\Omega|\leq6\alpha^{-1}.
\end{equation}
After changing $F_0,G_0$ on null sets, we may assume that both are
bounded by $\alpha$ everywhere.

It follows that
\begin{equation}\label{eq:good-good}
 T(F_0,G_0)(x)\leq2\delta\alpha^2=\lambda/4.
\end{equation}
Translation invariance and Tonelli's theorem give
\begin{align}
 \norm{T(F_0,G_b)}_1
 &\leq\alpha\int_{-\delta}^{\delta}
       \int_{\R}G_b(x+\chi(t))\,dx\,dt
 \leq2\delta\alpha,\label{eq:mixed}\\
 \norm{T(F_b,G_0)}_1&\leq2\delta\alpha.\notag
\end{align}
It remains to prove
\begin{equation}\label{eq:bad-bad}
 \int_{\Omega^c}T(F_b,G_b)(x)\,dx\leq C\alpha.
\end{equation}

\smallskip\noindent
\emph{Estimate for a pair of intervals.}
For $I\in\mathcal I$ and $J\in\mathcal J$, write
\[
 d_{IJ}=\dist(I,J),\qquad
 L_{IJ}=\max\{\ell(I),\ell(J)\},\qquad
 D=\max_{|t|\leq\delta}\psi(t)>0.
\]
Consider
\begin{equation}\label{eq:change}
 u=x+\phi(t),\qquad v=x+\chi(t).
\end{equation}
The Jacobian determinant is $\psi'(t)$. By
\eqref{eq:fold-derivative}, $\psi'$ does not vanish on
$(-\delta,0)$ or $(0,\delta)$ and has constant sign on each interval.
The map is therefore one-to-one on each branch: $v-u=\psi(t)$
determines $t$, and then $x=u-\phi(t)$.

If $x\notin\Omega$, $u\in I$, and $v\in J$, then
\[
 |u-x|\geq\ell(I),\qquad |v-x|\geq\ell(J).
\]
Indeed, $x$ lies outside both $3I$ and $3J$.
Together with \eqref{eq:fold-linear}, these inequalities imply
$|t|\geq L_{IJ}/C_0$. Moreover,
\[
 d_{IJ}\leq|v-u|=\psi(t)\leq C_0t^2,
\]
so $|t|\geq(d_{IJ}/C_0)^{1/2}$. Consequently,
\begin{equation}\label{eq:pair}
 \int_{\Omega^c}T(F_I,G_J)(x)\,dx
 \leq C\frac{a_Ib_J}{\max\{L_{IJ},\sqrt{d_{IJ}}\}}.
\end{equation}
For completeness, apply the change of variables separately on
$(-\delta,-\varepsilon)$ and $(\varepsilon,\delta)$. On the relevant
part of each branch,
\[
 \frac1{|\psi'(t)|}
 \leq\frac{C}{\max\{L_{IJ},\sqrt{d_{IJ}}\}}.
\]
The image lies in $I\times J$, and the change of variables has
multiplicity at most two.
Integration, followed by monotone convergence as
$\varepsilon\downarrow0$, proves \eqref{eq:pair}.
If $d_{IJ}>D$, the integral on the left vanishes.

\smallskip\noindent
\emph{Summation over interval pairs.}
By positivity and \eqref{eq:pair}, it suffices to show
\begin{equation}\label{eq:sum-goal}
 \sum_{\substack{I\in\mathcal I,\ J\in\mathcal J\\d_{IJ}\leq D}}
 \frac{a_Ib_J}{\max\{L_{IJ},\sqrt{d_{IJ}}\}}
 \leq C\alpha.
\end{equation}
First consider pairs with $d_{IJ}<L_{IJ}$ and
$\ell(J)\leq\ell(I)$. Every such $J$ is contained in $5I$:
its distance from $I$ is less than $\ell(I)$, and its length is at
most $\ell(I)$. Since the intervals of $\mathcal J$ are disjoint,
\[
 \sum_{\substack{J\in\mathcal J\\
           \ell(J)\leq\ell(I),\ d_{IJ}<\ell(I)}}b_J
 \leq2\alpha\sum_{J\subset5I}\ell(J)
 \leq10\alpha\ell(I).
\]
The denominator in \eqref{eq:sum-goal} is at least $\ell(I)$,
so the contribution of these pairs is bounded by
$10\alpha\sum_Ia_I\leq10\alpha$.
Interchanging $I$ and $J$ bounds the remaining pairs with
$d_{IJ}<L_{IJ}$ by the same quantity.

For the pairs with $L_{IJ}\leq d_{IJ}\leq D$, define
\[
 \overline F=\sum_{I\in\mathcal I}\frac{a_I}{\ell(I)}\1_I,
 \qquad
 \overline G=\sum_{J\in\mathcal J}\frac{b_J}{\ell(J)}\1_J.
\]
These functions satisfy
\begin{equation}\label{eq:bar-bounds}
 \norm{\overline F}_1,\norm{\overline G}_1\leq1,
 \qquad
 \norm{\overline F}_\infty,\norm{\overline G}_\infty\leq2\alpha.
\end{equation}
For each such pair and all $u\in I$, $v\in J$,
\[
 d_{IJ}\leq|u-v|
 \leq d_{IJ}+\ell(I)+\ell(J)\leq3d_{IJ}\leq3D.
\]
Let
\[
 K(s)=|s|^{-1/2}\1_{\{0<|s|\leq3D\}}(s).
 \qquad\text{Then}\qquad \norm{K}_1=4\sqrt{3D}.
\]
It follows that
\[
 \frac{a_Ib_J}{\sqrt{d_{IJ}}}
 \leq\sqrt3\int_I\int_J
 \overline F(u)\overline G(v)K(u-v)\,dv\,du.
\]
The rectangles $I\times J$ are pairwise disjoint up to boundaries.
Summing and using \eqref{eq:bar-bounds} yields
\begin{align*}
 \sum_{L_{IJ}\leq d_{IJ}\leq D}
 \frac{a_Ib_J}{\max\{L_{IJ},\sqrt{d_{IJ}}\}}
 &\leq\sqrt3\int_{\R}\int_{\R}
       \overline F(u)\overline G(v)K(u-v)\,dv\,du\\
 &\leq\sqrt3\norm{\overline F}_1
       \norm{\overline G}_\infty\norm{K}_1
 \leq C\alpha.
\end{align*}
This proves \eqref{eq:sum-goal} and hence \eqref{eq:bad-bad}.

Finally, expand $T(F,G)$ into its four nonnegative terms.
By \eqref{eq:good-good}, if $T(F,G)>\lambda$, the sum of the
other three terms exceeds $3\lambda/4$. Applying Chebyshev's
inequality on $\Omega^c$ and using
\eqref{eq:exceptional}, \eqref{eq:mixed}, and \eqref{eq:bad-bad},
we obtain
\[
 |\{T(F,G)>\lambda\}|
 \leq\frac6\alpha+\frac{4}{3\lambda}(4\delta\alpha+C\alpha)
 \leq C\lambda^{-1/2}.
\]
Rescaling the inputs proves \eqref{eq:fold-weak}.
\end{proof}

\section{The weak estimate for the circle}\label{sec:weak}

\begin{proof}[Proof of \eqref{eq:weak-distribution}]
It suffices to take $f,g\geq0$. Fix $\delta=1/10$ and let
\[
 \theta_0=3\pi/4,\qquad\theta_1=7\pi/4,\qquad
 U_j=(\theta_j-\delta,\theta_j+\delta)\quad(j=0,1).
\]
Let $\A_j$ denote the integral in \eqref{eq:operator} over $U_j$
for $j=0,1$, and let $\A_{\mathrm{reg}}$ denote the integral over
$E=[0,2\pi)\setminus(U_0\cup U_1)$.

For $a=1/\sqrt2$ and $\theta=\theta_0+t$, set
\begin{equation}\label{eq:circle-phi}
 \phi(t)=a(\cos t+\sin t-1),\qquad
 \psi(t)=2a(1-\cos t),\qquad
 \chi(t)=\phi(t)+\psi(t).
\end{equation}
The identities
\[
 -\cos(\theta_0+t)=a+\phi(t),\qquad
 -\sin(\theta_0+t)=-a+\chi(t)
\]
show that
\[
 \A_0(f,g)(x)
 =\int_{-\delta}^{\delta}F(x+\phi(t))G(x+\chi(t))\,dt,
\]
where $F(s)=f(s+a)$ and $G(s)=g(s-a)$.
For $|t|\leq\delta$, elementary trigonometric estimates give
\begin{equation}\label{eq:circle-geometry}
 |\phi(t)|+|\chi(t)|\leq4|t|,\qquad
 \frac a2t^2\leq\psi(t)\leq at^2,\qquad
 a|t|\leq|\psi'(t)|\leq2a|t|.
\end{equation}
Indeed, $|\sin t|\leq|t|$, $1-\cos t\leq t^2/2$,
$1-\cos t\geq t^2/4$, and $|\sin t|\geq|t|/2$ on this
interval. Lemma~\ref{lem:fold} and translation invariance of the
$L^1$ norm therefore imply
\begin{equation}\label{eq:arc-weak}
 |\{\A_0(f,g)>\lambda\}|
 \leq C\lambda^{-1/2}\norm{f}_1^{1/2}\norm{g}_1^{1/2}.
\end{equation}
For $f^\vee(s)=f(-s)$ and $g^\vee(s)=g(-s)$, the identity
\[
 \A_1(f,g)(x)=\A_0(f^\vee,g^\vee)(-x)
\]
gives the same estimate for $\A_1$.

We next prove
\begin{equation}\label{eq:regular-strong}
 \norm{\A_{\mathrm{reg}}(f,g)}_{L^{1/2}}
 \leq C\norm{f}_1\norm{g}_1.
\end{equation}
On $E$,
\[
 |\cos\theta+\sin\theta|\geq\sqrt2\sin\delta>0.
\]
The map
\[
 (x,\theta)\longmapsto(u,v)
 =(x-\cos\theta,x-\sin\theta)
\]
has Jacobian determinant $-\cos\theta-\sin\theta$ and
multiplicity at most two: $v-u=\cos\theta-\sin\theta$ has at
most two solutions in $[0,2\pi)$, and $x$ is then determined.
For $Q_k=[k,k+1)$ and $I_k=[k-1,k+2]$, both $u$ and $v$ belong
to $I_k$ whenever $x\in Q_k$. The change of variables, applied
on the finitely many monotonicity intervals, gives
\begin{equation}\label{eq:regular-local}
 \int_{Q_k}\A_{\mathrm{reg}}(f,g)(x)\,dx
 \leq C\left(\int_{I_k}f\right)\left(\int_{I_k}g\right).
\end{equation}
Write $f_k=\int_{I_k}f$ and $g_k=\int_{I_k}g$.
Applying Cauchy--Schwarz first on $Q_k$ and then in $k$, we obtain
\begin{align*}
 \int_{\R}\A_{\mathrm{reg}}(f,g)(x)^{1/2}\,dx
 &\leq C\sum_{k\in\mathbb Z}\sqrt{f_kg_k}\\
 &\leq C\left(\sum_kf_k\right)^{1/2}
            \left(\sum_kg_k\right)^{1/2}
 \leq C\norm{f}_1^{1/2}\norm{g}_1^{1/2}.
\end{align*}
The last inequality follows from the overlap bound
$\sum_k\1_{I_k}\leq4$. Squaring proves
\eqref{eq:regular-strong}. Chebyshev's inequality applied to
$\A_{\mathrm{reg}}^{1/2}$ gives its weak estimate.

Finally,
\[
 \{\A(f,g)>\lambda\}
 \subseteq\{\A_0(f,g)>\lambda/3\}
 \cup\{\A_1(f,g)>\lambda/3\}
 \cup\{\A_{\mathrm{reg}}(f,g)>\lambda/3\}.
\]
Combining the estimates for the three pieces proves
\eqref{eq:weak-distribution}. For general $f,g$, apply the result
to $|f|,|g|$.
The measure of $\{\A(|f|,|g|)>n\}$ tends to zero as
$n\to\infty$, so $\A(|f|,|g|)$ is finite almost everywhere.
Absolute convergence and
$|\A(f,g)|\leq\A(|f|,|g|)$ then give \eqref{eq:weak-norm}.
\end{proof}

\section{Failure of the strong endpoint}\label{sec:strong}

\begin{proof}[Proof of the remaining assertions of Theorem~\ref{thm:main}]
Let $a=1/\sqrt2$ and define
\begin{equation}\label{eq:h}
 h(s)=\frac{\1_{(0,e^{-4})}(s)}{s(\log(1/s))^{3/2}},
 \qquad f(u)=h(u-a),\qquad g(v)=h(v+a).
\end{equation}
Then
\begin{equation}\label{eq:h-norm}
 \norm{f}_1=\norm{g}_1=\int_0^{e^{-4}}
 \frac{ds}{s(\log(1/s))^{3/2}}
 =\int_4^\infty w^{-3/2}\,dw=1.
\end{equation}

Take $0<r<e^{-10}$ and $x=-r$. Restrict the angular integral to
$\theta=3\pi/4+t$ with $0<t<1/10$.
With $\phi$ and $\psi$ as in \eqref{eq:circle-phi},
\begin{equation}\label{eq:counter-arguments}
 f(-r-\cos\theta)=h(\phi(t)-r),\qquad
 g(-r-\sin\theta)=h(\phi(t)-r+\psi(t)).
\end{equation}
For $0<t<1/10$,
\begin{equation}\label{eq:phi-positive}
 \frac a2\leq\phi'(t)=a(\cos t-\sin t)\leq a,
 \qquad \frac a2t\leq\phi(t)\leq at.
\end{equation}
The lower bound for $\phi'$ follows from $\cos t\geq1-t^2/2$
and $\sin t\leq t$, while the upper bound follows from
$\cos t-\sin t\leq1$. Integrating gives the bounds for $\phi$.

Set $y=\phi(t)-r$ and restrict to $0<y<r^2$.
This restriction is valid since $\phi(1/10)\geq a/20$ and
$r+r^2<2e^{-10}<1/40<a/20$.
Since $\phi$ is strictly increasing, the inverse $t=t(y,r)$ is
well defined. Moreover,
\eqref{eq:phi-positive} implies
\begin{equation}\label{eq:t-range}
 \frac r a<t<\frac{4r}{a},\qquad dt\geq\frac{dy}{a}.
\end{equation}
Using \eqref{eq:circle-geometry} and putting
$z=y+\psi(t(y,r))$, we obtain
\begin{equation}\label{eq:z-range}
 ar^2=\frac{r^2}{2a}\leq z
 \leq\left(1+\frac{16}{a}\right)r^2<24r^2.
\end{equation}
Here $a^2=1/2$ and $\sqrt2<23/16$ justify the last bound.
In particular, $0<z<e^{-4}$, since $24e^{-20}<e^{-4}$.

Write $L=\log(1/r)>10$. Since $z\geq ar^2$,
\[
 \log(1/z)\leq2L+\log(1/a)\leq3L.
\]
Using $z\leq24r^2$ and $\log(1/z)\leq3L$ in \eqref{eq:h} gives
\begin{equation}\label{eq:h-z}
 h(z)\geq\frac{1}{24\cdot3^{3/2}\,r^2L^{3/2}}.
\end{equation}
Also $r^2<e^{-4}$ and
\begin{equation}\label{eq:h-tail}
 \int_0^{r^2}h(y)\,dy
 =\int_{2L}^\infty w^{-3/2}\,dw
 =\sqrt2\,L^{-1/2}.
\end{equation}
Combining \eqref{eq:counter-arguments}--\eqref{eq:h-tail} yields
\begin{align}
 \A(f,g)(-r)
 &\geq\frac1a\int_0^{r^2}h(y)h(y+\psi(t(y,r)))\,dy\notag\\
 &\geq\frac{1}{36\sqrt3}\,
       \frac{1}{r^2(\log(1/r))^2}.
       \label{eq:counter-lower}
\end{align}
Consequently,
\[
 \int_{\R}\A(f,g)(x)^{1/2}\,dx
 \geq\frac1{6\,3^{1/4}}
       \int_0^{e^{-10}}\frac{dr}{r\log(1/r)}=\infty.
\]
The last integral diverges after the substitution $w=\log(1/r)$,
which turns it into $\int_{10}^\infty dw/w$.
This proves \eqref{eq:strong-infinite}.

To obtain smooth inputs, choose a nondecreasing
$\eta\in C^\infty(\R)$ with $0\leq\eta\leq1$,
$\eta(s)=0$ for $s\leq1$, and $\eta(s)=1$ for $s\geq2$.
For integers $n$ large enough that $4/n<e^{-4}$, define
\[
 h_n(s)=
 \begin{cases}
 \displaystyle
 \frac{\eta(ns)\eta(n(e^{-4}-s))}
 {s(\log(1/s))^{3/2}},&0<s<e^{-4},\\[6pt]
 0,&\text{otherwise}.
 \end{cases}
\]
Then $h_n\in C_c^\infty(\R)$, $0\leq h_n\uparrow h$
pointwise, and $m_n:=\norm{h_n}_1\uparrow1$, with $m_n>0$.
By monotone convergence, first in $\theta$ and then in $x$,
\[
 \int_{\R}
 \A\bigl(h_n(\,\cdot-a),h_n(\,\cdot+a)\bigr)(x)^{1/2}\,dx
 \longrightarrow\infty.
\]
Set $f_n=m_n^{-1}h_n(\,\cdot-a)$ and
$g_n=m_n^{-1}h_n(\,\cdot+a)$. Then $f_n,g_n$ are nonnegative
functions in $C_c^\infty(\R)$ with $L^1$ norm one, and
$\norm{\A(f_n,g_n)}_{L^{1/2}}\to\infty$.
\end{proof}

\enlargethispage{2\baselineskip}

\end{document}